\documentclass[11pt]{article}
\usepackage{amsmath}
\usepackage{amssymb}
\usepackage{amstext}
\usepackage{amsthm}
\usepackage{url}

\newcommand{\SECTION}[1]{\S\ref{#1}}

\newcounter{defcounter}
\newtheorem{theorem}{Theorem}[section]

\newtheorem{lemma}[theorem]{Lemma}

\newtheorem{definition}[theorem]{Definition}

\newtheorem{algorithm}[theorem]{Algorithm}

\numberwithin{equation}{section}

\newcommand{\IGNORE}[1]{}
\newcommand{\DEF}[1]{\stackrel{\sf def}{#1}}

\def\Z{{\mathbb Z}}
\def\F{{\mathbb F}}
\def\G{G_\alpha}

\DeclareMathOperator{\ORD}{ord}

\DeclareMathOperator{\POLY}{poly}

\newcommand{\GCD}[2]{\big(#1,\; #2\big)}
\newcommand{\GCDrth}[1]{\GCD{#1}{x^r-\beta}}

\newcommand{\SET}[1]{\left\{ #1 \right\}}

\newcommand{\SETR}[2]{\left\{ #1 \;:\; #2 \right\}}

\begin{document}
\title{
On Taking $r$-th Roots without $r$-th Nonresidues over Finite Fields
and Its Applications
}
\author{Tsz-Wo Sze ({\tt szetszwo@cs.umd.edu})}
\date{Preliminary version, \today}
\maketitle
\begin{abstract}
We first show a deterministic algorithm for taking $r$-th roots over $\F_q$
without being given any $r$-th nonresidue,
where $\F_q$ is a finite field with $q$ elements
and $r$ is a small prime
such that $r^2$ divides of $q-1$.
As applications,
we illustrate deterministic algorithms over $\F_q$
for constructing $r$-th nonresidues,
constructing primitive elements,
solving polynomial equations
and computing elliptic curve ``$n$-th roots'',
and a deterministic primality test
for the generalized Proth numbers.
All algorithms are proved without assuming any unproven hypothesis.
They are efficient only if all the factors of $q-1$ are small and
some primitive roots of unity can be constructed efficiently over $\F_q$.
In some cases,
they are the fastest among the known deterministic algorithms.
\end{abstract}
\section{Introduction}\label{sect-introduction}
Let $\F_q$ be a finite field with $q$ elements
and $r$ be a prime.
Similar to the relationship between 
\emph{taking square roots} and \emph{constructing quadratic nonresidues}
over $\F_q$,
\emph{taking $r$-th roots} over $\F_q$,
for $r$ a divisor of $q-1$,
is polynomial-time equivalent to
\emph{constructing $r$-th nonresidues} over $\F_q$.
Clearly,
if $r$-th roots can be computed efficiently,
an $r$-th nonresidue can be constructed by
taking $r$-th roots repeatedly on a non-zero, non-identity element.
For the converse,
Tonelli-Shanks square root algorithm \cite{Tonelli1891, Shanks1972}
can be generalized to take $r$-th root,
provided that an $r$-th nonresidue is given as an input.

Without an $r$-th nonresidue as an input,
there is no known unconditionally deterministic polynomial-time
$r$-th root algorithms
over finite fields in general
except for some easy cases such as
\[
\GCD{r}{q-1}=1
\qquad\text{or}\qquad
r\;\parallel\;q-1;
\]
see \cite{Barreto2006}.
Under the assumption of the extended Riemann hypothesis,
Buchmann and Shoup showed a deterministic polynomial-time algorithm
for constructing $k$-th power nonresidues
over finite fields \cite{Buchmann1996}.

For \emph{taking square roots} over $\F_q$,
if a quadratic nonresidue is given,
we may use deterministic polynomial-time square root algorithms such as
Tonelli-Shanks \cite{Tonelli1891, Shanks1972}, 
Adleman-Manders-Miller \cite{Adleman1977}
and Cipolla-Lehmer \cite{Cipolla1903, Lehmer1969}.
Without quadratic nonresidues,
we have Schoof's square root algorithm over prime fields \cite{Schoof1985},
and our square root algorithm over any finite field \cite{stw2011sqrt}.
Note that these two algorithms run in polynomial-time only in some cases.
Obviously,
taking square roots and \emph{solving quadratic equations}
are polynomial-time equivalent.

A general problem is \emph{solving polynomial equations} over $\F_q$,
which is a generalization of the following problems,
\begin{itemize}
\item taking $r$-th roots,
\item constructing primitive $r$-th roots of unity,
\item constructing $r$-th nonresidues,
\item constructing primitive elements (generators of $\F_q^\times$),
\end{itemize}
where $r$ is a prime divisor of $q-1$.
It is clear that 
a primitive $r$-th root of unity can be computed efficiently 
from any $r$-th nonresidue.
By definition,
a primitive element is also an $r$-th nonresidue.

A more general problem is \emph{polynomial factoring} over $\F_q$.
Although there is a deterministic polynomial-time algorithm,
the celebrated Lenstra-Lenstra-Lov\'asz algorithm,
for factoring polynomials over rational numbers
\cite{Lenstra1982},
there are no known unconditionally finite field counterparts in general.
For deterministic polynomial factoring over finite fields,
we have Berlekamp's algorithm,
which is efficient only for $q$ small \cite{Berlekamp1967}.
For $q$ large,
there are probabilistic algorithms such as
the probabilistic version of Berlekamp's algorithm~\cite{Berlekamp1970},
Cantor and Zassenhaus~\cite{Cantor1981},
von zur Gathen and Shoup~\cite{Gathen1992},
and Kaltofen and Shoup~\cite{Kaltofen1998}.
Under some generalizations of Riemann hypothesis,
there is a subexponential-time algorithm by Evdokimov
for any finite field~\cite{Evdokimov1994},
and there are deterministic polynomial-time algorithms for some special cases.
For a survey,
see \cite{Gathen2001}.

The problem of solving polynomial equations is to find solutions of 
\[f(x)=0\]
over $\F_q$,
where $f(x)\in \F_q[x]$ is a polynomial.
Without loss of generality,
we may assume $f$ is a product of distinct linear factors
because \emph{squarefree factorization}
and \emph{distinct degree factorization} can be computed efficiently;
see \cite{Knuth1997, Gathen2001, Yun1976}.
If $f$ has a multiple root,
then 
\begin{eqnarray}\label{eqn-squarefree factorization}
\GCD{f}{f'}
\end{eqnarray}
 is a non-trivial factor of $f$,
where $f'$ denotes the derivative of $f$.
Since $x^q-x$ is the product of all monic linear polynomials in $\F_q[x]$,
the non-linear factors can be removed by computing
\begin{eqnarray}\label{eqn-distinct degree factorization}
\GCD{f(x)}{x^q-x}.
\end{eqnarray}

Let $E(\F_q)$ be an elliptic curve defined over $\F_q$.
An analogy of taking $r$-th roots over $\F_q$
is \emph{taking ``$n$-th root'' over $E(\F_q)$}.
Consider the following:
given a point $Q\in E(\F_q)$ and a positive integer $n$,
\renewcommand{\labelenumi}{(\roman{enumi})}\begin{enumerate}
\item[(E1)]
decide whether
\begin{eqnarray}\label{eqn-Q=nP}
Q &=& n P
\end{eqnarray}
for some $\infty\neq P\in E(\F_q)$;
\item[(E2)]
find $P$ if such $P$ exists.
\end{enumerate}\renewcommand{\labelenumi}{\arabic{enumi}.}
Note that,
when $Q=\infty$,
the trivial solution $P=\infty$ is excluded.
Although usually the elliptic curve group operation is written additively,
the nature of the problems above is closer to 
finite field $n$-th root
than finite field multiplicative inverse.

In this paper,
the main results are presented in \SECTION{sect-main results}.
We extend the ideas in \cite{stw2011sqrt}
to design a deterministic $r$-th root algorithm in \SECTION{sect-r th root}.
Then,
we demonstrate applications on primality testing,
solving polynomial equations
and taking elliptic curve ``$n$-th roots''
in
\SECTION{sect-primality testing},
\SECTION{sect-solving polynomial equations}
and \SECTION{sect-elliptic curve n th root},
respectively.

\section{Main Results}\label{sect-main results}
The main results are summarized by the theorems at the end of the section.
All theorems can be proved without assuming any unproven hypothesis.

All running times are given in term of bit operations.
We ignore logarithmic factors in running time
and adopt the $\tilde{O}(\;\cdot\;)$ notation.
Polynomial multiplication,
division with remainder,
greatest common divisor
over $\F_q$ can be computed 
using \emph{fast Fourier transforms} and other fast methods
in
\[
\tilde{O}(d\log q)
\]
bit operations for degree $d$ polynomials.
See \cite{Knuth1997} and \cite{Gathen2003}.

Let
\begin{eqnarray}\label{eqn-q=r_1^e_1...r_m^e_m t+1}
q &=& r_1^{e_1}\cdots r_m^{e_m}t+1
\end{eqnarray}
where
$r_1,\ldots,r_m$ are distinct primes
and $e_1,\ldots,e_m,t\geq 1$
such that $(r_1\cdots r_m,t)=1$.
Define sets of prime powers as follow.

\begin{definition}\label{def-Q}
Let ${\cal Q}_t$ be a set of prime powers.
For all $q\in{\cal Q}_t$,
$q$ can be written as the form in equation~(\ref{eqn-q=r_1^e_1...r_m^e_m t+1})
such that
\begin{eqnarray*}
r_1+\cdots+r_m+t &=& O(\POLY(\log q))
\end{eqnarray*}
and,
for $1\leq j\leq m$,
a primitive $z_j$-th root of unity $\zeta_{z_j}\in\F_q$
can be computed in polynomial-time,
where
\begin{eqnarray}\label{eqn-z_j}
z_j &\DEF{=}& \begin{cases}
4, & \text{if }r_j=2; \\
r_j, & \text{otherwise.}
\end{cases}
\end{eqnarray}
\end{definition}
Informally,
for $q\in {\cal Q}_t$,
$t$ and all the prime factors of $q-1$ are small
and a primitive $z_j$-th root of unity over $\F_q$ can be computed efficiently
for any prime factor $r_j$ of $(q-1)/t$.
Note that the factorization of $q-1$ can be computed efficiently in this case.
Denote the union of ${\cal Q}_t$ for $t\geq 1$ by
\begin{eqnarray}\label{eqn-overline F}
\overline{\cal Q} &\DEF{=}& \bigcup_{t\geq 1}{\cal Q}_t.
\end{eqnarray}
The main results are summarized below.

\begin{theorem}\label{thm-r th root}
Let $q\in\overline{\cal Q}$.
For $r\in\SET{r_1,\ldots,r_m}$,
there is a deterministic polynomial-time algorithm 
computing an $r$-th root of any $r$-th residue over $\F_q$.
Equivalently,
there is a deterministic polynomial-time algorithm 
constructing an $r$-th nonresidue over $\F_q$.
\end{theorem}

\begin{theorem}\label{thm-primitive element}
Let $q\in{\cal Q}_1$.
There is a deterministic polynomial-time algorithm 
constructing a primitive element over $\F_q$.
\end{theorem}

\begin{definition}\label{def-generalized proth}
A generalized Proth number is a positive integer of the form
\begin{eqnarray}\label{eqn-generalized Proth number}
N &=& r^e t+1
\end{eqnarray}
for prime $r$,
positive integers $e$ and $t$
such that $r^e>t$.
\end{definition}

\begin{theorem}\label{thm-r th root primality testing}
Let $N$ be a generalized Proth number.
There is a deterministic algorithm,
which runs in
\[
\tilde{O}((r(r+t)+\log N)r\log^2 N)
\]
bit operations
for deciding the primality of $N$.
Further,
if $r$ is a small constant and $t=O(\log N)$,
the running time is
\[
\tilde{O}(\log^3 N)
\]
bit operations.
\end{theorem}

\begin{theorem}\label{thm-solving polynomial equation}
Let $q\in{\cal Q}_1$.
There is a deterministic 
\[
\tilde{O}(\POLY(d\log q))
\]
algorithm to solve polynomial equation $f(x)=0$ over $\F_q$
for any degree $d$ polynomial $f(x)\in\F_q[x]$.
\end{theorem}

\begin{theorem}\label{thm-elliptic curve n th root}
Let $q\in{\cal Q}_1$.
There is a deterministic polynomial-time algorithm 
computing elliptic curve ``$n$-th roots'' over $\F_q$
for any positive integer $n=O(\POLY(\log q))$.
\end{theorem}
\section{Taking $r$-th Roots}\label{sect-r th root}
Let $\F_q$ be a finite field with $q$ elements.
Suppose 
\begin{eqnarray}
\beta &=& \alpha^r\in\F_q
\end{eqnarray}
for some $\alpha\in\F_q$ and some integer $r>1$.
The problem of \emph{taking $r$-th roots} over $\F_q$
is to find $\alpha$, 
given a finite field $\F_q$,
an element $\beta$ and an integer $r$.
If $r$ does not divide $q-1$,
the problem is easy.
If $r$ is a composite number,
we may first compute $\gamma$,
an $n$-th root of $\beta$ for $n$ a prime factor of $r$,
and then compute an $(r/n)$-th root of $\gamma$ to obtain $\alpha$.
Therefore,
assume that $r$ is a prime divisor of $q-1$.

The problem of taking $r$-th roots is reduced
to \emph{finding a non-trivial factor of $x^r-\beta$} over $\F_q$.
We label the following input items
and then show Algorithm~\ref{alg-rth root of beta} below.
\begin{description}
\item[(F):]
$\F_q$,
which is a finite field with $q$ elements.
\item[(R):]
$r$,
which is a prime divisor of $q-1$.
\item[(B):]
$\beta$,
which is an $r$-th residue in $\F_q$.
\end{description}

\begin{algorithm}[Compute an $r$-th root of $\beta$]\label{alg-rth root of beta}
The inputs are 
the ones specified in (F), (R) and (B);
and $f(x)$,
where $f(x)\in\F_q[x]$ is a monic non-trivial factor of $x^r-\beta$.
This algorithm returns an $r$-th root of $\beta$.
\end{algorithm}
\begin{enumerate}
\item
Let $n=\deg f$
and $c_0\in\F_q$ be the constant term of $f(x)$,
\item
Find integers $u,v$ by the Euclidean algorithm such that $un+vr=1$.
\item
Return $(-1)^{nu}c_0^u\beta^v$.
\end{enumerate}

\begin{lemma}\label{lem-non-trivial factor <=> r th root}
Algorithm~\ref{alg-rth root of beta} is correct.
\end{lemma}
\begin{proof}
Let $\rho$ be a primitive $r$-th root of unity in $\F_q$.
Since 
\[
x^r-\beta=\prod_{j=0}^{r-1}(x-\rho^j\alpha),
\]
we have $c_0=(-1)^n\rho^k\alpha^n$ for some integer $k$.
We also have $(n,r)=1$ because $0<n<r$ and $r$ is a prime.
There exist integers $u,v$ such that $un+vr=1$.
Finally,
\[
(-1)^{nu}c_0^u\beta^v=\rho^{ku}\alpha,
\]
is an $r$-th root of $\beta$.
The lemma follows.
\end{proof}

\subsection{Find a Non-trivial Factor of $x^r-\beta$}
We extend the square root algorithm in \cite{stw2011sqrt}
to show a deterministic algorithm,
Algorithm~\ref{alg-factor(x^r-beta)},
for finding a non-trivial factor of $x^r-\beta$.
Unlike other algorithms,
such as the generalized Shanks's algorithm,
Algorithm~\ref{alg-factor(x^r-beta)}
does not require any $r$-th nonresidue as an input
and the associated proofs do not assume any unproven hypothesis.
Similar to \cite{stw2011sqrt},
Algorithm~\ref{alg-factor(x^r-beta)}
requires finding primitive roots of unity.
It is obvious that \emph{finding an $N$-th primitive root of unity}
is not harder than \emph{finding an $N$-th nonresidue}
because,
given an $N$-th nonresidue,
an $N$-th primitive root of unity can be easily computed.
Below are some known cases that
primitive roots of unity can be computed efficiently;
see \cite{stw2011sqrt} for more details.
Let $p$ be the characteristic of $\F_q$.
Denote a fixed primitive $k$-th root of unity in $\F_q$ by $\zeta_k$.
\renewcommand{\labelenumi}{(\roman{enumi})}\begin{enumerate}
\item
$\zeta_2$ or $\zeta_3$
when $p\equiv 1\pmod{12}$.
\item
$\zeta_{2\cdot3^n+1}$ for $n\geq 1$
when $2\cdot3^n+1$ is a prime and $p\equiv 13,25\pmod{36}$.
\item
$\zeta_r$
when $q=r^e t+1$ with $t$ small.
\end{enumerate}\renewcommand{\labelenumi}{\arabic{enumi}.}

The arithmetic of the square root algorithm in \cite{stw2011sqrt}
is carried out over a specially constructed group,
$\G$,
which is isomorphic to $\F_q^\times$ and a degenerated elliptic curve.
\emph{Taking square root} is obviously equivalent to
\emph{finding a non-trivial factor of $x^2-\beta$}.
It is possible to formulate the algorithm in \cite{stw2011sqrt}
so that the arithmetic is carried out over the ring $\F_q[x]/(x^2-\beta)$
for factoring the polynomial $x^2-\beta$.
We generalize this idea and work on the ring $\F_q[x]/(x^r-\beta)$
in Algorithm~\ref{alg-factor(x^r-beta)}.
When $r=2$,
Algorithm~\ref{alg-factor(x^r-beta)} and
the algorithm in \cite{stw2011sqrt} are essentially the same.

The ``problem'' of working on the ring $\F_q[x]/(x^r-\beta)$ is that 
there are zero divisors.
However, 
if we have a zero divisor $f(x)$,
then 
\[
\GCDrth{f(x)}
\]
is a non-trivial factor of $x^r-\beta$.
This idea is similar to Lenstra's elliptic curve integer factoring algorithm
\cite{Lenstra1987}.
He works on the ring $\Z/n\Z$ for some composite integer $n$,
try to find a zero divisor $z$ in $\Z/n\Z$
and then $(z,n)$ is a non-trivial factor of $n$.

If $q-1$ is not divisible by $r^2$,
it is easy to compute $\alpha$.
Thus,
assume 
\begin{equation}\label{eqn-r^2|(q-1)}
r^2\mid(q-1).
\end{equation}
As in equation~(\ref{eqn-q=r_1^e_1...r_m^e_m t+1}),
write
\[
q-1=r_1^{e_1}\cdots r_m^{e_m}t.
\]
Without loss of generality,
assume $r_1=r$.
Note that $e_1\geq 2$ by assumption~(\ref{eqn-r^2|(q-1)}).
Once the $r_j$'s are fixed,
the partial factorization of $q-1$ can be computed easily.
Algorithm~\ref{alg-factor(x^r-beta)} applies to any finite field
but it is efficient only if $q\in\overline{\cal Q}$;
see definition~(\ref{eqn-overline F}).
We present Algorithm~\ref{alg-factor(x^r-beta)} below
and discuss the details in the following sections.
Note that it returns immediately
once Algorithms~\ref{alg-find a},
\ref{alg-find ell} or \ref{alg-find k_0}
have returned a non-trivial factor of $x^r-\beta$.
\begin{description}
\item[(R'):]
$r$ which satisfies (R) and (\ref{eqn-r^2|(q-1)}).
\item[(Q):]
$r_1,\ldots,r_m$, $e_1,\ldots,e_m$ and $t$
such that $r_1=r$
and $q-1=r_1^{e_1}\cdots r_m^{e_m}t$ is the partial factorization
satisfied equation~(\ref{eqn-q=r_1^e_1...r_m^e_m t+1}).
\end{description}

\begin{algorithm}[Find a non-trivial factor of $x^r-\beta$]
\label{alg-factor(x^r-beta)}
The inputs are
the ones specified in assumptions~(Q), (F), (R') and (B).
This algorithm returns a non-trivial factor of $x^r-\beta$.
\end{algorithm}
\begin{enumerate}
\item[I:]
If $r=2$ and $\beta=1$,
return $x+1$. \\
If $r=2$ and $\beta=-1$,
return $x+\sqrt{-1}$. \\
Otherwise,
compute $\rho=\zeta_r$, a primitive $r$-th root of unity.
\item[II:]
Find $a$ by Algorithm~\ref{alg-find a} with $k=rt$.
\item[III:]
Find $\ell$ by Algorithm~\ref{alg-find ell}.
\item[IV:]
Find $k_0$ by Algorithm~\ref{alg-find k_0}.
\item[V:]
Find a non-trivial factor $f(x)$ of $x^r-\beta$
by Algorithm~\ref{alg-split x^r-beta}. \\
Return $f(x)$.
\end{enumerate}

\subsubsection{Find $a$ such that $\GCDrth{g_{a,k}(x)}=1$}
For $a\in\F_q$, $a^r\neq\beta$,
define a rational function 
\begin{eqnarray}\label{eqn-psi}
\psi_a(x) &\DEF{=}& \frac{a-x}{a-\rho x}\in\F_q(x).
\end{eqnarray}
For $0\leq i<r$,
let 
\begin{eqnarray}
\label{eqn-c_i}
c_i &\DEF{=}& \psi_a(\rho^i\alpha)\in\F_q^\times; \\
\label{eqn-d_i}
d_i &\DEF{=}& \ORD c_i,
\end{eqnarray}
the order of $c_i$ over $\F_q^\times$.
In other words,
we have
\begin{eqnarray*}
\psi_a(x) &\equiv& c_i\pmod{x-\rho^i\alpha}; \\
\psi_a(x)^{d_i} &\equiv& 1\pmod{x-\rho^i\alpha}.
\end{eqnarray*}
Instead of working with the rational function $\psi_a$ directly,
define polynomials,
\begin{eqnarray}
\label{eqn-g_k(x,y,z)}
g_k(x,y,z) &\DEF{=}& (y-x)^k-z(y-\rho x)^k\in\F_q[x,y,z]; \\
\label{eqn-g_(a,k)(x)}
g_{a,k}(x) &\DEF{=}& g_k(x,a,1)\in\F_q[x],
\end{eqnarray}
for $k>0$.
We have the following lemma.
\begin{lemma}\label{lem-d_i cases}
Let $k$ be a positive integer.
\renewcommand{\labelenumi}{(\arabic{enumi})}\begin{enumerate}
\item
$d_i$ divides $k$ for all $0\leq i<r$
if and only if
\begin{eqnarray*}
g_{a,k}(x) &\equiv& 0\pmod{x^r-\beta}.
\end{eqnarray*}
\item
There exists $i$, $j$ 
such that $d_i$ divides $k$ but $d_j$ does not divide $k$
if and only if
\[
\GCDrth{g_{a,k}(x)}
\]
is a non-trivial factor of $x^r-\beta$.
\item
$d_i$ does not divide $k$ for all $0\leq i<r$
if and only if
\begin{eqnarray*}
\GCDrth{g_{a,k}(x)} &=& 1.
\end{eqnarray*}
\end{enumerate}\renewcommand{\labelenumi}{\arabic{enumi}.}
\end{lemma}
\begin{proof}
It is straightforward.
\end{proof}
For the cases in Lemma~\ref{lem-d_i cases},
case~(1) is not useful to our algorithm.
We show in the lemma below that
the number of possible values of $a$'s 
falling into this case is bounded above by $k$.
If case~(2) occurs,
we are done.
Otherwise,
we find an $a$ falling into case~(3)
in Algorithm~\ref{alg-find a}.

\begin{lemma}\label{lem-at most k a}
There are at most $k$ distinct $a\in\F_q$ 
such that $a^r\neq\beta$
and 
\begin{eqnarray*}
g_{a,k}(x) &\equiv& 0 \pmod{x^r-\beta}.
\end{eqnarray*}
\end{lemma}
\begin{proof}
Suppose that,
for $1\leq i\leq k+1$,
we have $a_i\in\F_q$,
$a_i^r\neq \beta$
and $g_{a_i,k}(x)\equiv 0 \pmod{x^r-\beta}$.
Then, $g_{a_i,k}(\alpha)=0$ and so $\psi_{a_i}(\alpha)^k=1$.
Since $\psi_a(\alpha)\neq \psi_b(\alpha)$ whenever $a\neq b$,
there are $k+1$ distinct elements in $\F_q$
such that the multiplicative orders of all these elements divide $k$.
It is a contradiction.
The lemma follows.
\end{proof}

We show Algorithm~\ref{alg-find a} below.
Note that the $\rho$,
which is computed in Algorithm~\ref{alg-factor(x^r-beta)} Step~I,
is used for computing $g_{a_i,k}(x)$ in II.2.
\begin{description}
\item[(Z):]
$\rho$,
where $\rho=\zeta_r\in\F_q$ is a primitive $r$-th root of unity.
\end{description}

\begin{algorithm}[Find $a$]
\label{alg-find a}
The inputs are
the ones specified in (F), (R'), (B) and (Z);
and $k$, where $k>1$ is an integer.
This algorithm either
returns a non-trivial factor of $x^r-\beta$,
or returns $a\in\F_q$
such that 
\begin{equation}\label{eqn-a}
a^r\neq \beta
\qquad\text{and}\qquad
\GCDrth{g_{a,k}(x)}=1.
\end{equation}
\end{algorithm}
\begin{enumerate}
\item[II:]
Consider $k+1$ distinct elements $a_1,\ldots,a_{k+1}\in\F_q$.
\begin{enumerate}
\item[II.1:]
If there exists $i$ such that $a_{i}^r=\beta$, \\
return $x-a_i$.
\item[II.2:]
If there exists $i$ such that 
$f(x)=\GCDrth{g_{a_i,k}(x)}$ \\
is a non-trivial factor $x^r-\beta$, \\
return $f(x)$.
\item[II.3:]
Set $a=a_j$ for some $1\leq j\leq k+1$ 
such that $\GCDrth{g_{a_j,k}(x)}=1$. \\
Return $a$.
\end{enumerate}
\end{enumerate}
\begin{lemma}\label{lem-find a}
Algorithm~\ref{alg-find a} is correct.
\end{lemma}
\begin{proof}
The algorithm is obviously correct if it returns at II.1 or II.2.
Otherwise,
there exist $1\leq j\leq k+1$ such that $\GCDrth{g_{a_j,k}(x)}=1$
by Lemma \ref{lem-at most k a}.
The lemma follows.
\end{proof}

\subsubsection{Find $\ell=r_{j_0}$ such that $\GCDrth{g_{a,h_{j_0}}(x)}=1$}
Let
\begin{eqnarray*}
h_j &\DEF{=}& \begin{cases}
(q-1)/r^{e_1-1} & \text{, if }j=1; \\
(q-1)/r_j^{e_j} & \text{, otherwise.}
\end{cases}
\end{eqnarray*}
Algorithm~\ref{alg-find a} is executed with $k=rt$
in Algorithm~\ref{alg-factor(x^r-beta)} Step~II.
Algorithm~\ref{alg-find ell} is shown below.
\begin{description}
\item[(A):]
$a$,
where $a\in\F_q$ satisfies condition~(\ref{eqn-a}) with $k=rt$.
\item[(L):]
$\ell$,
where $\ell=r_{j_0}$ for some $1\leq j_0\leq m$
such that
\begin{eqnarray*}
\GCDrth{g_{a,h_{j_0}}(x)} &=& 1.
\end{eqnarray*}
\end{description}

\begin{algorithm}[Find $\ell$]
\label{alg-find ell}
The inputs are 
the ones specified in (Q), (F), (R'), (B), (Z) and (A).
This algorithm either returns a non-trivial factor of $x^r-\beta$,
or returns an integer $\ell$ satisfying (L).

\end{algorithm}
\begin{enumerate}
\item[III.1:]
If there exist $1\leq j\leq m$ such that $f(x)=\GCDrth{g_{a,h_j}(x)}$ \\
is a non-trivial factor of $x^r-\beta$, \\
return $f(x)$.
\item[III.2:]
Set $\ell=r_{j_0}$ for some $1\leq j_0\leq m$
such that $\GCDrth{g_{a,h_{j_0}}(x)}=1$. \\
Return $\ell$.
\end{enumerate}

\begin{lemma}\label{lem-find ell}
Algorithm~\ref{alg-find ell} is correct.
\end{lemma}
\begin{proof}
The algorithm is obviously correct if it returns at III.1.
Otherwise,
$\GCDrth{g_{a,h_j}(x)}$ for $1\leq j\leq m$ are trivial factors of $x^r-\beta$.

Suppose,
for all $1\leq j\leq m$,
\begin{eqnarray*}
g_{a,h_j}(x) &\equiv& 0 \pmod{x^r-\beta}.
\end{eqnarray*}
Then,
\[
\psi_a(x)^{h_1}
\equiv\cdots\equiv
\psi_a(x)^{h_m}\equiv 1\pmod{x^r-\beta},
\]
or equivalently,
\begin{eqnarray*}
\psi_a(\rho^i\alpha)^{h_j} &=& 1
\end{eqnarray*}
for all $0\leq i<r$ and all $1\leq j\leq m$.
Recall that $d_i$ is
the multiplicative order of $\psi_a(\rho^i\alpha)$ in $\F_q^\times$
defined in equation~(\ref{eqn-d_i}).
For all $0\leq i<r$ and all $1\leq j\leq m$,
we have
\[
d_i\mid h_j.
\]
Since $rt=\gcd(h_1,\ldots,h_m)$,
we have
\[
d_i\mid r t.
\]
It is a contradiction
because $d_i$ does not divide $rt$ for all $0\leq i<r$
by assumption~(A) and Lemma~\ref{lem-d_i cases} case (3).
The lemma follows.
\end{proof}

\subsubsection{Find $k_0$
such that $D_{k'}(x)=x^r-\beta$ for $0\leq k'\leq k_0$
and $D_{k''}(x)=1$ for $k_0<k''\leq e'$}
Let
\begin{eqnarray}\label{eqn-e'}
e' &\DEF{=}& \begin{cases}
e_1 - 1, & \text{if } \ell=r; \\
e_{j_0}, & \text{otherwise}.
\end{cases}
\end{eqnarray}
Define polynomials
\begin{eqnarray}\label{eqn-D_k}
D_i(x) &\DEF{=}& \GCDrth{g_{a,(q-1)/\ell^i}(x)}\in\F_q[x]
\end{eqnarray}
for $0\leq i\leq e'$.
By Lemma~\ref{lem-d_i cases} case~(1) with $k=q-1$,
\begin{eqnarray*}
D_0(x) &=& x^r-\beta
\end{eqnarray*}
and,
by assumption (L),
\begin{eqnarray*}
D_{e'}(x) &=& 1.
\end{eqnarray*}
We show Algorithm~\ref{alg-find k_0} below.
\begin{description}
\item[(K):]
$k_0$
such that,
for all $0\leq k'\leq k_0$ and all $k_0<k''\leq e'$,
\[
D_{k'}(x) = x^r-\beta
\qquad\text{and}\qquad
D_{k''}(x) = 1.
\]
\end{description}

\begin{algorithm}[Find $k_0$]
\label{alg-find k_0}
The inputs are 
the ones specified in (Q), (F), (R'), (B), (Z), (A) and (L).
This algorithm either returns a non-trivial factor of $x^r-\beta$,
or returns an integer $k_0$ satisfying (K).
\end{algorithm}
\begin{enumerate}
\item[IV.1:]
Compute $D_k(x)$ by definition~(\ref{eqn-D_k}) for all $0\leq k\leq e'$.
\item[IV.2:]
If there exist $0<k<e'$ \\
such that $D_{k}(x)$ is a non-trivial factor of $x^r-\beta$, \\
return $D_{k}(x)$.
\item[IV.3:]
Set $k_0$ to be the largest $k$ such that $D_k(x)=x^r-\beta$. \\
Return $k_0$.
\end{enumerate}
\begin{lemma}\label{lem-find k_0}
Algorithm~\ref{alg-find k_0} is correct.
\end{lemma}
\begin{proof}
The algorithm is obviously correct if it returns at IV.2.
Suppose all $D_k(x)$ are trivial factors of $x^r-\beta$.
By Lemma~\ref{lem-D_k(x)=x^r-beta for 0<=k<=k'} below,
there exists $0\leq k_0<e'$ satisfying (K).
The lemma follows.
\end{proof}
\begin{lemma}\label{lem-D_k(x)=x^r-beta for 0<=k<=k'}
If $D_i(x)=x^r-\beta$ for some $0\leq i<e'$,
then $D_{k'}(x)=x^r-\beta$ for all $0\leq k' \leq i$.
\end{lemma}
\begin{proof}
It follows from the case (1) of Lemma \ref{lem-d_i cases}.
\end{proof}

\subsubsection{Split $x^r-\beta$}
Equipped with conditions (A), (L) and (K),
we are ready to split $x^r-\beta$.
Below is the key lemma.

\begin{lemma}\label{lem-n_i != n_j}
Let $N>1$ be a prime power such that $N\neq r$.
Let $D$ be a positive integer.
Suppose,
for $0\leq i<r$, 
\begin{eqnarray*}
\psi_a(\rho^i\alpha)^D &=& \zeta_N^{n_i}
\end{eqnarray*}
for some integer $n_i\in(\Z/N\Z)^\times$,
where $a\in\F_q$ such that $a^r\neq\beta$,
and $\zeta_N$ is a primitive $N$-th root of unity.
There exist $i$ and $j$ such that
\begin{eqnarray*}
n_i &\neq& n_j.
\end{eqnarray*}
\end{lemma}
\begin{proof}
Suppose 
\[
n_0=\cdots=n_{r-1}=n
\]
for some integer $n$ with $(n,N)=1$.
Let $\zeta=\zeta_N^n$.
We have 
\[
\psi_a(\alpha)^D
=\psi_a(\rho\alpha)^D
=\cdots
=\psi_a(\rho^{r-1}\alpha)^D
=\zeta,
\]
which is equivalent to
\[
g_D(\alpha,a,\zeta)
=g_D(\rho\alpha,a,\zeta)
=\cdots
=g_D(\rho^{r-1}\alpha,a,\zeta)
=0.
\]
By definition (\ref{eqn-g_k(x,y,z)}),
\begin{eqnarray*}
g_D(\rho^i\alpha,a,\zeta) &=& (a-\rho^i\alpha)^D-\zeta(a-\rho^{i+1}\alpha)^D.
\end{eqnarray*}
Then,
\[
(a-\alpha)^D(1-\zeta^r)
\;\;=\;\; \sum_{i=0}^{r-1}\zeta^ig_D(\rho^i\alpha,a,\zeta)
\;\;=\;\; 0.
\]
Thus, 
$\zeta^r=1$ since $a\neq\alpha$.
It is a contradiction because $N$ does not divide $r$.
The lemma follows.
\end{proof}

We show Algorithm~\ref{alg-split x^r-beta} below.
Define
\begin{eqnarray*}
d &\DEF{=}& \begin{cases}
(q-1)/\ell^{k_0+2}, & \text{if } \ell=r; \\
(q-1)/\ell^{k_0+1}, & \text{otherwise}.
\end{cases}
\end{eqnarray*}

\begin{algorithm}[Split $x^r-\beta$]\label{alg-split x^r-beta}
The inputs are
the ones specified in (Q), (F), (R'), (B), (Z), (A), (L) and (K).
In addition,
assume $r\neq2$ when $\beta^r=1$.
This algorithm returns a non-trivial factor of $x^r-\beta$.
\end{algorithm}
\begin{enumerate}
\item[V.1:]
Case $\ell\neq r$:
\begin{enumerate}
\item[V.1.1:]
Compute $\zeta_\ell$,
a primitive $\ell$-th root of unity.
\item[V.1.2:]
For each $0<n<\ell$, \\
compute $f_n(x)=\GCDrth{g_d(x,a,\zeta_\ell^n)}$, \\
return $f_{n}(x)$ if $f_n(x)$ is a non-trivial factor of $x^r-\beta$.
\end{enumerate}
\item[V.2:]
Case $\ell=r$ and $\beta^r\neq 1$:
\begin{enumerate}
\item[V.2.1:]
Compute $\zeta_{r^2}$,
a primitive $r^2$-th root of unity,
recursively.
In other words,
use Algorithm~\ref{alg-rth root of beta}
and Algorithm~\ref{alg-factor(x^r-beta)} with $\beta=\zeta_r$.
\item[V.2.2:]
For each $n\in(\Z/r^2\Z)^\times$, \\
compute $f_n(x)=\GCDrth{g_d(x,a,\zeta_{r^2}^n)}$, \\
return $f_{n}(x)$ if $f_n(x)$ is a non-trivial factor of $x^r-\beta$.
\end{enumerate}
\item[V.3:]
Case $\ell=r\neq 2$ and $\beta^r=1$:
\begin{enumerate}
\item[V.3.1:]
For each $n\in(\Z/r^2\Z)^\times$, \\
compute $f_n(x)=\GCDrth{g_d(x,a,x^n)}$, \\
return $f_{n}(x)$ if $f_n(x)$ is a non-trivial factor of $x^r-\beta$.
\end{enumerate}
\end{enumerate}

\begin{lemma}\label{lem-split x^r-beta V.1}
Algorithm~\ref{alg-split x^r-beta} Step~V.1 is correct.
\end{lemma}
\begin{proof}
Recall that $d_i$ is defined in equation~(\ref{eqn-d_i}).
For all $0\leq i<r$,
we have
\[
d_i\mid\ell d
\qquad\text{and}\qquad
d_i\nmid d
\]
by assumption~(K).
Since $\F_q^\times$ is cyclic,
\begin{eqnarray*}
\psi_a(\rho^i\alpha)^d &=& \zeta_\ell^{n_i}
\end{eqnarray*}
for some $0<n_i<\ell$.
Let $g_n(x)\DEF{=}g_d(x,a,\zeta_\ell^n)\in\F_q[x]$.
Then,
\begin{eqnarray*}
g_{n_0}(x) &\equiv& 0 \pmod{x-\alpha}.
\end{eqnarray*}
By Lemma \ref{lem-n_i != n_j} with $N=\ell$ and $D=d$,
there exists $0<j<r$ such that
\begin{eqnarray*}
g_{n_0}(x) &\not\equiv& 0 \pmod{x-\rho^j\alpha}.
\end{eqnarray*}
Therefore,
$\GCDrth{g_{n_0}(x)}$ is a non-trivial factor of $x^r-\beta$.
The remaining question is how to find $n_0$?
It is not required.
For $0<n<\ell$,
compute 
\[
\GCDrth{g_n(x)}
\]
in order to find a non-trivial factor of $x^r-\beta$.
The lemma follows.
\end{proof}

\begin{lemma}\label{lem-split x^r-beta V.2}
Algorithm~\ref{alg-split x^r-beta} Step~V.2 is correct.
\end{lemma}
\begin{proof}
Similar to the proof of the previous lemma,
for all $0\leq i<r$,
\[
d_i\mid r^2 d
\qquad\text{and}\qquad
d_i\nmid r d
\]
by assumption (K).
We have
\begin{eqnarray*}
\psi_a(\rho^i\alpha)^d &=& \zeta_{r^2}^{n_i}
\end{eqnarray*}
for some $n_i\in(\Z/r^2\Z)^\times$.
Let $g_n(x)\DEF{=}g_d(x,a,\zeta_{r^2}^n)\in\F_q[x]$.
Then,
\begin{eqnarray*}
g_{n_0}(x) &\equiv& 0 \pmod{x-\alpha}; \\
g_{n_0}(x) &\not\equiv& 0 \pmod{x-\rho^j\alpha}
\end{eqnarray*}
for some $0<j<r$
by Lemma \ref{lem-n_i != n_j} with $N=r^2$ and $D=d$.
For each $n\in(\Z/r^2\Z)^\times$,
compute 
\[
\GCDrth{g_n(x)}
\]
to find a non-trivial factor of $x^r-\beta$.
The lemma follows.
\end{proof}

In the case $\ell=r$,
a primitive $r^2$-th root of unity, 
$\zeta_{r^2}$,
is required.
Interestingly,
$\zeta_{r^2}$ can be computed recursively
--- by taking $r$-th root of $\rho$,
or equivalently,
by finding a non-trivial factor of $x^r-\rho$.
Execute Algorithm~\ref{alg-factor(x^r-beta)} with $\beta=\rho$
and denote the output of Step~III by $\ell'$.
If $\ell'\neq r$,
we proceed with Step~V.1.
Otherwise,
we have $\ell'=r$.
Then,
\begin{equation}\label{eqn-(g_d(x,a,zeta_(r^2)^n), x^r-rho)}
\GCD{g_d(x,a,\zeta_{r^2}^n)}{x^r-\rho}
\end{equation}
is a non-trivial factor of $x^r-\rho$ for some $n$.
Nevertheless,
the gcd cannot be computed directly because $\zeta_{r^2}$ is not available.
The idea is to replace $\zeta_{r^2}$ with $x$.
In other words,
use $g_d(x,a,x^n)$,
instead of $g_d(x,a,\zeta_{r^2}^n)$,
in (\ref{eqn-(g_d(x,a,zeta_(r^2)^n), x^r-rho)}).
This idea does not work for the case $\ell=r=2$ and $\beta^r=1$,
which is handled separately in Step~I.
We have the following lemma.

\begin{lemma}\label{lem-g_d(x,a,x^n)}
Suppose $r$ is an odd prime.
If 
\begin{eqnarray*}
g_d(x,a,x^n) &\equiv& 0\pmod{x-\zeta_{r^2}}
\end{eqnarray*}
for some $n\in(\Z/r^2\Z)^\times$,
there exists $0<i<r$
such that 
\begin{eqnarray*}
g_d(x,a,x^n) &\not\equiv& 0 \pmod{x-\rho^i\zeta_{r^2}}.
\end{eqnarray*}
\end{lemma}
\begin{proof}
Let $\zeta=\zeta_{r^2}$.
Suppose 
\begin{eqnarray*}
g_d(x,a,x^n) &\equiv& 0\pmod{x-\rho^i\zeta}
\end{eqnarray*}
for all $0\leq i<r$.
Then,
\[
g_d\big(\zeta,a,\zeta^n\big)
=g_d\big(\rho\zeta,a,(\rho\zeta)^n\big)
=\cdots
=g_d\big(\rho^{r-1}\zeta,a,(\rho^{r-1}\zeta)^n\big)
=0.
\]
Let
\[
s_k
\;\;\DEF{=}\;\; \sum_{i=0}^{k-1}i
\;\;=\;\; k(k-1)/2.
\]
Note that $r$ divides $s_r$.
By definition (\ref{eqn-g_k(x,y,z)}),
\begin{eqnarray*}
g_d\big(\rho^i\zeta,a,(\rho^i\zeta)^n\big)
&=& (a-\rho^i\zeta)^d-\rho^{in}\zeta^n(a-\rho^{i+1}\zeta)^d.
\end{eqnarray*}
Then,
\begin{eqnarray*}
0
&=& \sum_{i=0}^{r-1}\rho^{s_in}\zeta^{in} g_d\big(\rho^i\zeta,a,(\rho^i\zeta)^n\big) \\
&=& (a-\zeta)^d(1-\rho^{s_rn}\zeta^{rn}) \\
&=& (a-\zeta)^d(1-\zeta^{rn}).
\end{eqnarray*}
Since $a\neq\zeta$,
we have $\zeta^{rn}=1$.
It is a contradiction.
The lemma follows.
\end{proof}

\begin{lemma}\label{lem-split x^r-beta V.3}
Algorithm~\ref{alg-split x^r-beta} Step~V.3 is correct.
\end{lemma}
\begin{proof}
Suppose,
for $0\leq i<r$,
\begin{eqnarray*}
\psi_a(\rho^i\zeta_{r^2})^d &=& \zeta_{r^2}^{n_i}
\end{eqnarray*}
for some integer $n_i\in(\Z/r^2\Z)^\times$.
Let $g_n(x)\DEF{=}g_d(x,a,x^n)\in\F_q[x]$.
Consider the polynomial $g_{n_0}(x)$.
We have 
\begin{eqnarray*}
g_{n_0}(x) &\equiv& 0\pmod{x-\zeta_{r^2}}; \\
g_{n_0}(x) &\not\equiv& 0\pmod{x-\rho^j\zeta_{r^2}}
\end{eqnarray*}
for some $0<j<r$
by Lemma~\ref{lem-g_d(x,a,x^n)}.
For each $n\in(\Z/r^2\Z)^\times$,
compute 
\[
\GCD{g_n(x)}{x^r-\rho}
\]
to find a non-trivial factor of $x^r-\rho$.
The lemma follows.
\end{proof}

\subsection{Running Time Analysis}\label{sect-rth root running time}
We analyze the running time of Algorithms~\ref{alg-rth root of beta}
and \ref{alg-factor(x^r-beta)} below.

\begin{lemma}\label{lem-running time rth root of beta}
Algorithm~\ref{alg-rth root of beta} runs in
\[
\tilde{O}(\log r\log q)
\]
bit operations.
\end{lemma}
\begin{proof}
The Euclidean algorithm can be executed in $\tilde{O}(\log r)$
and the last step can be evaluated in $\tilde{O}(\log r\log q)$.
The lemma follows.
\end{proof}

In Algorithm~\ref{alg-factor(x^r-beta)},
a common operation is to compute
\[
\GCDrth{g_k(x,y,z^n)}
\]
for some fixed $k>0$,
some fixed $N$ and all $1\leq n\leq N$, 
where $y\in\F_q$ and $z\in\F_q\cup\{x\}$.
We show the required running time below
and then show the running time of Algorithm~\ref{alg-factor(x^r-beta)}.

\begin{lemma}\label{lem-running time (g_k(x,y,z^n), x^r-beta)}
Let $k$ and $N$ be positive integers.
Given $y\in\F_q$, $z\in\F_q\cup\{x\}$ and $\rho$,
it takes 
\[
\tilde{O}((\log k + N)r\log q)
\]
bit operations to compute $\GCDrth{g_k(x,y,z^n)}$
for all $1\leq n\leq N$.
\end{lemma}
\begin{proof}
For any $a,b\in\F_q$,
the power-modulo $(a-bx)^k\pmod{x^r-\beta}$ can be computed in
$\tilde{O}(r\log k\log q)$.
Let
\begin{eqnarray*}
f_1(x) &\DEF{=}& (y-x)^k \pmod{x^r-\beta}; \\
f_2(x) &\DEF{=}& (y-\rho x)^k \pmod{x^r-\beta}.
\end{eqnarray*}
By equation~(\ref{eqn-g_k(x,y,z)}),
\begin{eqnarray*}
g_k(x,y,z^n) &\equiv& f_1(x)-z^nf_2(x)\pmod{x^r-\beta}.
\end{eqnarray*}
Once $f_1$ and $f_2$ are obtained,
the GCDs $\GCDrth{g_k(x,y,z^n)}$ for $1\leq n\leq N$
can be computed incrementally using $\tilde{O}(Nr\log q)$.
The lemma follows.
\end{proof}

Recall that $r_1=r$ by assumption (Q)
and $z_i$ is defined in equation~(\ref{eqn-z_j}).

\begin{lemma}\label{lem-running time factor(x^r-beta)}
Algorithm~\ref{alg-factor(x^r-beta)} is correct and runs in
\begin{equation}\label{eqn-factor(x^r-beta) running time}
\tilde{O}\left(
Z_{\max}
+ \left(r(r+t) + r_{\max} + m\log q\right)r\log q
\right)
\end{equation}
bit operations,
where 
\begin{eqnarray*}
r_{\max} &=& \max(r_1,\ldots,r_m), \\
Z_{\max} &=& \max(Z_{z_1},\ldots,Z_{z_m}),
\end{eqnarray*}
where $Z_n$ is the time required
for constructing a primitive $n$-th root of unity over $\F_q$.
\end{lemma}
\begin{proof}
If it returns at Step~I,
the algorithm is obviously correct.
Otherwise, the correctness follows from 
Lemmas~\ref{lem-find a},
\ref{lem-find ell},
\ref{lem-find k_0},
\ref{lem-split x^r-beta V.1},
\ref{lem-split x^r-beta V.2} and
\ref{lem-split x^r-beta V.3}.

We show the running time as follows.
Clearly,
Step~I requires
\[
O(Z_{z_1}).
\]
For each $a_i$,
the running times are $\tilde{O}(\log r\log q)$ in II.1
and $\tilde{O}(r\log k\log q)$ in II.2 and II.3.
Step~II requires 
\begin{eqnarray*}
\tilde{O}(kr\log q) &=& \tilde{O}(r^2t\log q)
\end{eqnarray*}
since there are $k+1$ elements and $k=rt$.
Step~III requires 
\[
\tilde{O}(mr\log^2 q).
\]
By first computing $D_{e'}(x)$ in $\tilde{O}(r\log^2 q)$,
then using the intermediate results 
to compute $D_{e'-1}(x)$ in $\tilde{O}(r\log \ell\log q)$
and so on,
Step~IV requires 
\[
\tilde{O}(r\log \ell\log^2 q).
\]

Suppose $\ell\neq r$ or $\beta^r=1$ in Step~V for the following.
We are either in V.1 or V.3.
V.1.1 requires $Z_\ell$ to compute $\zeta_\ell$.
By Lemma~\ref{lem-running time (g_k(x,y,z^n), x^r-beta)},
V.1.2 and V.3.1 can be done in $\tilde{O}((\log q + \ell)r\log q)$ 
and $\tilde{O}((\log q + r^2)r\log q)$,
respectively.
Step~V without V.2 takes
\[
\tilde{O}(Z_\ell + (r^2 + \ell + \log q)r\log q).
\]
The overall running time of the algorithm in this case is
(\ref{eqn-factor(x^r-beta) running time}).

Suppose $\ell=r\neq 2$ and $\beta^r\neq1$.
Everything remains the same except that we are in V.2.
By Lemma~\ref{lem-running time rth root of beta} and above,
 the recursive call in V.2.1 requires
(\ref{eqn-factor(x^r-beta) running time}).
V.2.2,
which is similar to V.3.1,
requires $\tilde{O}((\log q + r^2)r\log q)$.
The overall running time of the algorithm in this case is also
(\ref{eqn-factor(x^r-beta) running time}).

The lemma follows
\end{proof}

By the running time in (\ref{eqn-factor(x^r-beta) running time}),
Algorithm~\ref{alg-factor(x^r-beta)} is efficient
only if $t$ and all the prime factors of $q-1$ are small
and,
for all $1\leq i\leq m$,
a primitive root $z_i$-th of unity can be constructed efficiently over $\F_q$.

\begin{proof}[Proof of Theorem \ref{thm-r th root}]
If $r^2\nmid (q-1)$,
taking $r$-th roots over $\F_q$ can be easily done in polynomial-time.
Otherwise,
$r^2\mid (q-1)$.
Since $q\in\overline{\cal Q}$,
we have 
\begin{eqnarray*}
t+r_{\max}+Z_{\max} &=& O(\POLY(\log q)).
\end{eqnarray*}
Taking $r$-th roots for any $r$-th residue over $\F_q$ can be done
in polynomial-time by Lemmas \ref{lem-running time rth root of beta}
and \ref{lem-running time factor(x^r-beta)}.

For constructing an $r$-th nonresidue $\zeta_{r^{e_1}}\in\F_q$,
we begin with $\zeta_r$,
compute $\zeta_{r^2}=\sqrt[r]{\zeta_r}$,
then compute $\zeta_{r^3}=\sqrt[r]{\zeta_{r^2}}$ and so on.
The theorem follows.
\end{proof}

\begin{proof}[Proof of Theorem \ref{thm-primitive element}]
For any $q\in{\cal Q}_1$, for each $i$,
an $r_i$-th nonresidue $\zeta_{r_i^{e_i}}\in\F_q$ 
can be computed in deterministic polynomial by Theorem \ref{thm-r th root}.
The product $\prod_{i=1}^m\zeta_{r_i^{e_i}}$ 
is a primitive element over $\F_q$.
The theorem follows.
\end{proof}

We show an interesting special case below.
\begin{theorem}\label{thm-rth root, q=r^e t+1}
Let $q=r^e t+1$ be a prime power
for $r$ prime, $e>1$, $t\geq 1$ and $(r,t)=1$.
There is a deterministic algorithm,
which runs in
\[
\tilde{O}((r(r+t)+\log q)r\log q)
\]
bit operations for taking $r$-th root over $\F_q$.

Further,
there is a deterministic algorithm,
which runs in
\[
\tilde{O}((r(r+t)+\log q)r\log^2 q)
\]
bit operations for constructing an $r$-th nonresidue over $\F_q$.
\end{theorem}
\begin{proof}
Firstly,
find a primitive $r$-th root of unity, $\zeta_r$,
by \cite[Alg.~5.9]{stw2011sqrt} in $\tilde{O}((t+\log q)\log q)$.
Then,
use Algorithms~\ref{alg-rth root of beta} and \ref{alg-factor(x^r-beta)}
to compute an $r$-th root
in $\tilde{O}(\log r\log q)$ and $\tilde{O}((r(r+t)+\log q)r\log q)$,
respectively.
For constructing an $r$-th nonresidue,
it requires taking $O(\log q)$ $r$-th roots.
The theorem follows.
\end{proof}
\section{Primality Testing}\label{sect-primality testing}
Let $N$ be a generalized Proth number defined
in Definition~\ref{def-generalized proth}.
Consider the problem of deciding the primality of $N$.
In \cite{stw2011proth},
a deterministic primality test is created
from a deterministic square root algorithm and Proth's theorem;
see \cite{Williams1998} for the details of Proth's theorem.
The idea is generalized -- 
we design a deterministic primality test
using the deterministic $r$-th root algorithm presented
in \SECTION{sect-r th root}
and a generalized Proth's theorem (Theorem \ref{thm-generalized proth} below).
This generalization of Proth's theorem is well known.
The idea of our primality test is similar to Pocklington-Lehmer primality test;
see \cite[\S 7.2]{lcw2008}.
Theorem \ref{thm-r th root primality testing} is proved in the following.

\begin{proof}[Proof of Theorem \ref{thm-r th root primality testing}]
If $N$ is prime,
an $r$-th nonresidue $\zeta_{r^e}\in\Z/N\Z$
can be constructed in 
\[
\tilde{O}((r(r+t)+\log N)r\log^2 N)
\]
by Theorem~\ref{thm-rth root, q=r^e t+1}.
If $N$ is composite,
$\zeta_{r^e}\not\in\Z/N\Z$ 
by Theorem \ref{thm-generalized proth} below.
Since all algorithms, including Algorithm~5.9 in \cite{stw2011sqrt},
Algorithms~\ref{alg-rth root of beta} and \ref{alg-factor(x^r-beta)}
in the previous section,
are deterministic,
the primality of $N$ can be decided
by trying constructing an $r$-th nonresidue over the integer ring $\Z/N\Z$
using these algorithms.
The theorem follows.
\end{proof}

For $N=r^e t+1$ with $r$ a small constant and $t=\tilde{O}(\log N)$,
the running time of our primality test is 
\[
\tilde{O}(\log^3 N).
\]
It is faster than all known deterministic tests.
The running time of the AKS test \cite{Agrawal2004} 
and Lenstra-Pomerance's modified AKS test \cite{Lenstra2009}
are $\tilde{O}(\log^{7.5} N)$ and $\tilde{O}(\log^6 N)$, 
respectively.
Assuming the Extended Riemann Hypothesis,
Miller's test \cite{Miller1975} is deterministic 
with running time $\tilde{O}(\log^4N)$.

We will use the following lemma to prove Theorem \ref{thm-generalized proth}.
Denote Euler's function by $\phi(\;\cdot\;)$.

\begin{lemma}\label{lem-r^e|phi(l^k) => k=1}
Let $n=\ell^k$ be a prime power for some prime $\ell$ and $k\geq 1$.
Let $r^e$ be a prime power with $r\neq\ell$.
If 
\[
r^e\mid\phi(n)
\qquad\text{and}\qquad
r^e>\sqrt{n},
\]
then $k=1$ and $n$ is a prime.
\end{lemma}
\begin{proof}
We have
\begin{eqnarray*}
\phi(n) &=& (\ell-1)\ell^{k-1}.
\end{eqnarray*}
Then,
$r^e$ divides $(\ell-1)$ and so $\ell>r^e$.
If $k>1$,
then 
\[
\phi(n)\geq(\ell-1)\ell>r^{2e}>n,
\]
which is a contradiction.
Thus,
$k=1$ and $n$ is a prime.
\end{proof}

\begin{theorem}\label{thm-generalized proth}
{\bf (Generalized Proth's Theorem)}
Let $N=r^e t+1$ be a generalized Proth number defined
in Definition~\ref{def-generalized proth}
for prime $r$,
positive integers $e$ and $t$
such that $r^e>t$.
If
\begin{equation}\label{eqn-generalized proth}
a^{N-1}\equiv1\pmod{N}
\qquad\text{and}\qquad
a^{(N-1)/r}\not\equiv 1\pmod{N},
\end{equation}
for some integer $a$,
then $N$ is a prime.
\end{theorem}
\begin{proof}
It is easy to see that, for any generalized Proth number,
\begin{eqnarray*}
r^e &>& \sqrt{N}.
\end{eqnarray*}
Suppose there exists an integer $a$ 
satisfying equations~(\ref{eqn-generalized proth}).
Let $d\DEF{=}\ORD_N a$ be the order of $a$ in $(\Z/N\Z)^\times$.
Then,
$r^e$ divides $d$
and so 
\[
r^e\mid\phi(N).
\]
If $N=\ell^k$ for some prime $\ell$ and $k\geq 1$,
then $N$ is a prime
by Lemma \ref{lem-r^e|phi(l^k) => k=1}.

Suppose $N=\ell_1^{k_1}\cdots\ell_m^{k_m}$
for $m>1$,
some distinct primes $\ell_1,\ldots,\ell_m$
and some integers $k_1,\ldots,k_m\geq 1$.
Let $b\equiv a^{d/r^e}\pmod{N}$.
Then,
\[
\ORD_N b=r^e.
\]
Let $d_i$ be the order of $b$ in $(\Z/\ell_i^{k_i}\Z)^\times$.
Since
\begin{eqnarray*}
b^{r^e} &\equiv& 1\pmod{\ell_i^{k_i}}
\end{eqnarray*}
for all $1\leq i\leq m$,
\[
d_i \;\DEF{=}\; \ORD_{\ell_i^{k_i}} b \;=\; r^{s_i}
\]
for some $0\leq s_i\leq e$.
Without loss of generality,
assume $s_1\geq s_i$ for all $1\leq i\leq m$.
Then,
\begin{eqnarray*}
b^{d_1} &\equiv& 1\pmod{\ell_i^{k_i}}
\end{eqnarray*}
for all $1\leq i\leq m$.
By the Chinese Remainder Theorem,
\begin{eqnarray*}
b^{d_1} &\equiv& 1\pmod{N}.
\end{eqnarray*}
Therefore,
$r^e$ divides both $d_1$ and $\phi(\ell_1^{k_1})$.
By Lemma \ref{lem-r^e|phi(l^k) => k=1} with $n=\ell_1^{k_1}$,
we have $k_1=1$.
Write 
\begin{eqnarray*}
\ell_1=r^et_1+1
\qquad\text{and}\qquad
N/\ell_1 &=& r^{e_0}t_0+1
\end{eqnarray*}
with $(r,t_0)=1$.
Since $\ell_1(N/\ell_1)=N=r^e t+1$,
we have
\begin{eqnarray*}
t &=& t_0t_1r^{e_0}+t_1+t_0r^{e_0-e}.
\end{eqnarray*}
Then,
$e_0\geq e$,
otherwise,
$t$ is not an integer.
However,
\[
N
\;\;=\;\; \ell_1(N/\ell_1)
\;\;>\;\; r^{e+e_0}
\;\;\geq\;\; r^{2e}
\;\;>\;\; N.
\]
which is a contradiction.
The theorem follows.
\end{proof}
\section{Solving Polynomial Equations}
\label{sect-solving polynomial equations}
Let $\F_q$ be the finite field of $q$ elements.
Let $f(x)\in\F_q[x]$ be a polynomial.
In this section,
we consider the problem of solving the polynomial equation 
\[
f(x)=0,
\]
over $\F_q$.
By (\ref{eqn-squarefree factorization})
and (\ref{eqn-distinct degree factorization})
in \SECTION{sect-introduction},
we may assume $f$ is a product of distinct linear factors.
Without loss of generality,
assume $\deg f>1$ and $f(0)\neq0$.
When the prime factors of $q-1$ are small,
the problem of solving polynomial equations over $\F_q$
is polynomial-time reducible to 
the problem of taking $r$-th roots over $\F_q$
for all prime factors $r$ of $q-1$.

The idea is simple:
suppose $f(x)$ is a divisor of $x^d-a$
for some integer divisor $d$ of $q-1$
and some $d$-th residue $a\in\F_q$
with 
\[
\ORD(a)=(q-1)/d.
\]
Let $\ell$ be a prime factor of $d$
and $\zeta_\ell\in\F_q$ be a primitive $\ell$-th root of unity.
For $0\leq i<\ell$,
let
\begin{eqnarray*}
h_i(x) &\DEF{=}& x^{d/\ell}-\zeta_\ell^ia^{1/\ell}\in\F_q[x]; \\
g_i(x) &\DEF{=}& \GCD{f(x)}{h_i(x)}\in\F_q[x].
\end{eqnarray*}
We have
\begin{eqnarray*}
x^d-a &=& \prod_{i=0}^{\ell-1}h_i(x); \\
f(x) &=& \prod_{i=0}^{\ell-1}g_i(x).
\end{eqnarray*}
If $g_i$
is a non-trivial factor of $f$
for some $0\leq i<\ell$,
we are done
(or keep factoring until the complete factorization of $f$ is obtained.)
Otherwise,
$f$ is a divisor of $h_{i_0}$
for some $0\leq i_0<\ell$.
Repeat the process with $d'=d/\ell$ and $a'=\zeta_\ell^ia^{1/\ell}$.
Initially,
$f(x)$ is a divisor or $x^{q-1}-1$,
i.e.~$a=1$ and $d=q-1$.
We show a deterministic algorithm 
to find a non-trivial factor of $f$ below.

\begin{algorithm}[Factoring products of linear polynomials]
\label{alg-factoring products of linear polynomials}
The inputs are the prime factorization $q-1=r_1^{e_1}\cdots r_m^{e_m}$
and a polynomial $f(x)\in\F_q[x]$
such that $f(0)\neq 0$
and $f(x)$ is a product of two or more distinct monic linear polynomials.
This algorithm returns a non-trivial factor of $f$.
\end{algorithm}
\begin{enumerate}
\item[I:]
Set $a=1$ and $d=q-1$. \\
Compute $\zeta_{r_j}$ for $1\leq j\leq m$.
\item[II:]
For each $1\leq j\leq m$:
\begin{enumerate}
\item[II.1:]
For each $1\leq k\leq e_j$:
\begin{enumerate}
\item[II.1.1:]
Compute $b\in\F_q$ such that $b^{r_j}=a$ using some algorithm.
\item[II.1.2:]
Compute $g_i(x)=\GCD{f(x)}{x^{d/r_j}-\zeta_{r_j}^i b}$ for all $0\leq i<r_j$.
\item[II.1.3:]
If $g_i$ is a non-trivial factor of $f$ for some $0\leq i<r_j$, \\
return $g_i$. \\
Otherwise,
set $i_0=i$ such that $g_i=f$.
\item[II.1.4:]
Set $a=\zeta_{r_j}^{i_0}b$ and $d=d/r_j$.
\end{enumerate}
\end{enumerate}
\end{enumerate}

\begin{lemma}\label{lem-factoring products of linear polynomials}
Algorithm \ref{alg-factoring products of linear polynomials} is correct.
\end{lemma}
\begin{proof}
Clearly,
the loops maintain an invariant that
$a$ is an $r_j$-th residue over $\F_q$ at II.1.1.
Thus,
the $r_j$-th roots of $a$ are in $\F_q$.

We show by induction that 
$f(x)\mid(x^d-a)$ is an invariant at II.1.1.
When $j=k=1$,
we have $a=1$ and $d=q-1$.
By the input assumption,
$f(x)$ divides $x^{q-1}-1$.
Let $a_{j_0,k_0}$ and $d_{j_0,k_0}$ be the values of $a$ and $d$ at II.1.1
when $j=j_0$ and $k=k_0$.
Suppose $f(x)$ divides $x^{d_{j_0,k_0}}-a_{j_0,k_0}$.
Let
\begin{eqnarray*}
b_{j_0,k_0} &\DEF{=}& a_{j_0,k_0}^{1/r_{j_0}}\in\F_q; \\
h_{i,j_0,k_0}(x)
&\DEF{=}& x^{d_{j_0,k_0}/r_{j_0}}-\zeta_{r_{j_0}}^ib_{j_0,k_0}\in\F_q[x]; \\
g_{i,j_0,k_0}(x)
&\DEF{=}& \GCD{f(x)}{h_{i,j_0,k_0}(x)}\in\F_q[x].
\end{eqnarray*}
Then,
\begin{eqnarray*}
x^{d_{j_0,k_0}}-a_{j_0,k_0}
&=& \prod_{i=0}^{r_{j_0}-1} h_{i,j_0,k_0}(x); \\
f(x) &=& \prod_{i=0}^{r_{j_0}-1}g_{i,j_0,k_0}(x).
\end{eqnarray*}
If there exists $g_i$ a non-trivial factor of $f$, done.
Otherwise,
there exists a unique $i_0$ such that $g_{i_0}=f$.
Denote the pair of $j,k$ following $j_0,k_0$ by $j_1,k_1$.
When $j=j_1$ and $k=k_1$,
we have
\[
a=a_{j_1,k_1}=\zeta_{r_{j_0}}^{i_0}b_{j_0,k_0}
\qquad\text{and}\qquad
d=d_{j_1,k_1}=d_{j_0,k_0}/r_{j_0}
\] at II.1.1.
By the definition of $g_{i_0}$,
$f(x)$ divides $x^d-a$.

As a consequence,
$f(x)$ divides $x^d-a$ right after II.1.4.
The algorithm eventually returns a non-trivial factor of $f$ at II.1.3. 
Otherwise,
for $j=m$ and $k=e_m$,
we have $d=1$ right after II.1.4.
Then,
$f(x)$ divides a linear polynomial.
It is a contradiction.
The lemma follows.
\end{proof}

\begin{lemma}\label{lem-running time: factoring products of linear polynomials}
Algorithm \ref{alg-factoring products of linear polynomials} runs in
\[
\tilde{O}((Z_{\max} + R_{\max} + (r_{\max}+\log q)\deg f\log q)\log q)
\]
bit operations,
where 
\begin{eqnarray*}
r_{\max} &\DEF{=}& \max(r_1,\ldots,r_m), \\
Z_{\max} &\DEF{=}& \max(Z_{r_1},\ldots,Z_{r_m}), \\
R_{\max} &\DEF{=}& \max(R_{r_1},\ldots,R_{r_m}),
\end{eqnarray*}
where $Z_n$ and $R_n$ are respectively the time required
for constructing a primitive $n$-th root of unity
and computing an $n$-th root over $\F_q$.
\end{lemma}
\begin{proof}
Obviously, Step~I requires 
\[
O(mZ_{\max}).
\]
II.1.1 requires $O(R_{r_j})$.
In II.1.2,
first compute $h(x)\DEF{=}x^{d/r_j}\bmod f(x)$
using $\tilde{O}(\deg f\log^2 q)$
and then compute $\GCD{f(x)}{h(x)-\zeta_{r_j}^i b}$ for $0\leq i<r_j$
using $\tilde{O}(r_j\deg f\log q)$.
The time required for II.1.3 and II.1.4 are clearly dominated by II.1.2.
Since there are at most
\begin{eqnarray*}
\sum_{j=1}^m e_j &=& O(\log q)
\end{eqnarray*}
iterations,
Step~II requires 
\[
\tilde{O}((R_{\max}+(r_{\max}+\log q)\deg f\log q)\log q).
\]
The lemma follows.
\end{proof}

\begin{lemma}\label{lem-r th root => solving polynomial equation}
Let $\F_q$ be a finite field of $q$ elements.
For every prime factor $r$ of $q-1$,
suppose $r=O(\POLY(\log q))$
and
there are deterministic polynomial-time algorithms
for constructing $r$-th primitive root of unity
and computing $r$-th roots over $\F_q$.
Then,
there is a deterministic polynomial-time algorithm 
solving any polynomial equation over $\F_q$.
\end{lemma}
\begin{proof}
Without loss of generality,
assume the input polynomial $f(x)\in\F_q[x]$
is a product of two or more distinct monic linear polynomials
and $f(0)\neq0$.
The complete factorization of $f$ can be computed in polynomial-time
using Algorithm \ref{alg-factoring products of linear polynomials} repeatedly.
The overall running time is $\tilde{O}(\POLY(\deg f\log q))$
by Lemma \ref{lem-running time: factoring products of linear polynomials}.
Since the input size is $O(\deg f\log q)$,
it is a polynomial-time algorithm.
The lemma follows.

\end{proof}

\begin{proof}[Proof of Theorem \ref{thm-solving polynomial equation}]
Since $q\in{\cal Q}_1$,
the theorem is an obvious consequence of
Theorem~\ref{thm-r th root}
and Lemma~\ref{lem-r th root => solving polynomial equation}.
\end{proof}
\section{The Elliptic Curve ``$n$-th Root'' Problem}
\label{sect-elliptic curve n th root}
Let $\F_q$ be a finite field with $q$ elements.
For simplicity,
assume the characteristic of $\F_q$ is neither 2 nor 3.
Denote an elliptic curve $E$ over $\F_q$
by the \emph{Weierstrass equation}
\[E:y^2=x^3+a_4x+a_6
\]
for some $a_4,a_6\in\F_q$.
In the following,
we study the elliptic curve ``$n$-th root''
described in \SECTION{sect-introduction}.
Problems~(E1) and (E2) will be reduced
to the problem of solving polynomial equations.

It is well known that multiplication by $n$ over $E$ is an endomorphism,
\begin{eqnarray*}
n(x,y) &=& \left(\frac{U_1(x)}{V_1(x)},\;y\frac{U_2(x)}{V_2(x)}\right)
\end{eqnarray*}
for some polynomials $U_1(x),V_1(x),U_2(x),V_2(x)\in\F_q[x]$
such that
\begin{eqnarray*}
\deg U_1 &=& n^2, \\
\deg V_1 &\leq& n^2-1, \\
(U_1,V_1) &=& (U_2,V_2) \;=\; 1.
\end{eqnarray*}
All polynomials $U_1$, $V_1$, $U_2$ and $V_2$
can be computed in polynomial-time;
see \cite{lcw2008} for the details.

Suppose $Q\neq\infty$.
We have $Q=(a,b)$ for some $a,b\in\F_q$.
If $Q=n(x_0,y_0)$ for some $x_0,y_0\in\F_q$,
then $x_0$ is a solution of 
\begin{eqnarray*}
f(x) &\DEF{=}& U_1(x)-aV_1(x)=0
\end{eqnarray*}
over $\F_q$.
Suppose $\alpha_1,\ldots,\alpha_k\in\F_q$ are the roots of 
equation $f(x)=0$.
Let
\begin{eqnarray}\label{eqn-set of (alpha_i,beta)}
\label{eqn-g_i(y)}
g_i(y)
&\DEF{=}&
y^2-(\alpha_i^3+a_4\alpha_i+a_6); \\
\label{eqn-h_i(y)}
h_i(y)
&\DEF{=}&
y U_2(\alpha_i)-b V_2(\alpha_i); \\
{\bf P} &\DEF{=}& 
\SETR{(\alpha_i,\beta)\in\F_q^2}
{g_i(\beta)=0
\text{ and }h_i(\beta)=0}.
\end{eqnarray}
The set ${\bf P}$ is the complete set of solutions
of equation~(\ref{eqn-Q=nP}).
For (E1),
equation~(\ref{eqn-Q=nP}) has a solution
if and only if ${\bf P}$ is non-empty.
For (E2),
any point $P\in{\bf P}$ is a solution of equation~(\ref{eqn-Q=nP}).

Suppose $Q=\infty$.
Denote a fixed algebraic closure of $\F_q$ by $\overline{\F_q}$.
Let 
\begin{eqnarray*}
E[n](\F_q) &\DEF{=}& E[n]\cap E(\F_q),
\end{eqnarray*}
where $E[n]$ denotes the $n$-torsion subgroup of $E(\overline{\F_q})$.
Then
\[
P\in E[n](\F_q)
\]
if $P$ is a solution of equation~(\ref{eqn-Q=nP}).
Let $\alpha_1,\ldots,\alpha_k\in\F_q$ 
be the roots of the equation $V_1(x)=0$
and
\begin{eqnarray*}
{\bf P}' &\DEF{=}&
\SETR{(\alpha_i,\beta)_{1\leq i\leq k}}
{g_i(\beta)=0},
\end{eqnarray*}
where $g_i$ is defined in equation~(\ref{eqn-g_i(y)}).
Problems (E1) and (E2) can be solved similar to before.

\begin{proof}[Proof of Theorem \ref{thm-elliptic curve n th root}]
By the discussion above,
the sets ${\bf P}$ and ${\bf P}'$ can be computed
by solving a few polynomial equations over $\F_q$.
When $q\in{\cal Q}_1$,
a degree $d$ polynomial equation can be solved in $\tilde{O}(\POLY(d\log q))$
by Theorem \ref{thm-solving polynomial equation}.
Since $n=O(\POLY(\log q))$,
the degrees of all polynomials in the discussion above
are also $O(\POLY(\log q))$.
The theorem follows.
\end{proof}
Note that the running time of the elliptic curve $n$-th root algorithm
depends mostly on the finite field $\F_q$ but not the curve.
Once polynomial equations can be solved efficiently over $\F_q$,
elliptic curve $n$-th roots can be computed efficiently for any curve.
Also,
the number of points of $E(\F_q)$ is not required in the algorithm.
\IGNORE{
\paragraph{Acknowledgments:}
We thank Lawrence C. Washington for his guidance of the study.
He provided a lot of invaluable comments in this work.
}
\bibliographystyle{plain}
\bibstyle{plain}
\bibliography{rthroot}

\end{document}